\documentclass[11pt]{article}

\usepackage[T1]{fontenc}
\usepackage{textcomp}
\usepackage{newtxtext}

\makeatletter
\@ifundefined{cprime}{}{}
\@ifundefined{cdprime}{}{}
\@ifundefined{cyrdash}{}{}
\makeatother

\usepackage{geometry}
\usepackage{graphicx, color}
\usepackage[dvipsnames]{xcolor}
\usepackage{amsmath, amsfonts}
\usepackage{mathrsfs}
\usepackage{mathtools}
\usepackage{url}

\newtheorem{thm}{Theorem}[subsection]
\newtheorem{lmm}[thm]{Lemma}
\newtheorem{prp}[thm]{Proposition}

\newtheorem{rmk}[thm]{Remark}

\newtheorem{fct}[thm]{Fact}

\newenvironment{proof}[1][Proof]{%
	\par\noindent\textit{#1. }\ignorespaces
}{%
	\hfill$\square$\par
}

\newenvironment{namedthm}[1]
{\par\medskip\noindent\textbf{#1.}\itshape}
{\par\medskip}

\usepackage{hyperref}
\hypersetup{
	bookmarks=true,
	bookmarksdepth=tocdepth,
	bookmarksnumbered=true,
	colorlinks=true,
	linkcolor=blue,
	citecolor=teal,
	urlcolor=PineGreen,
	pdftitle={},
	pdfsubject={},
	pdfauthor={},
	pdfkeywords={},
}

\title{Explicit Formulas for the One-Parameter Group Generated by the Dunkl Operator on $\mathbb{R}$}
\author{Temma Aoyama}
\date{}

\begin{document}
	
	\maketitle
	\begin{abstract}
		Let $T_{b}$ be the Dunkl operator for the reflection group $G=\mathbb{Z}/2\mathbb{Z}$, and $D_{b}:=|x|^{b}\,T_{b}\,|x|^{-b}$. We compute explicitly the unitary one-parameter group $e^{tD_{b}}$ generated by $D_{b}$. We obtain two representations: a boundary value representation from the upper and lower half-planes, and a real-variable formula consisting of a translation term and a principal value integral term with an explicit kernel expressed in terms of Legendre functions.
	\end{abstract}

			\tableofcontents
			
			\section{Statement of the Main Result}
			
			The purpose of this paper is to obtain explicit formulas for the one-parameter unitary group generated by an operator $D_b$ on $L^2(\mathbb R)$; see Subsection~\ref{1.1} for the definition of $D_b$ and Fact~\ref{fct} for the basic properties.
			
			More precisely, we first derive a boundary value representation of $e^{tD_b}$, and then rewrite it as a real-variable formula consisting of a translation term and a principal value integral term with an explicit kernel expressed in terms of Legendre functions.
			The main results are presented in Subsection~\ref{1.2} as Theorems~\hyperref[mtA]{A}, \hyperref[mtB]{B}, and \hyperref[mtC]{C}.
			
			\subsection{Review of Background Material}\label{1.1}
			
			Suppose $b> -\frac{1}{2}$. We define the operator $D_{b}$ on  $|x|^{b}\mathcal{S}(\mathbb{R})\left(\subset L^{2}(\mathbb{R})\right)$ by
			$$D_{b}f(x) :=\frac{d f}{d x}(x)-\frac{b}{x}f(-x).$$
			This coincides with the Dunkl operator \cite{MR951883} for a reflection group $G=\mathbb{Z}/2\mathbb{Z}$, via conjugation by $|x|^{b}$; that is, $T_{b}f(x) := |x|^{-b}D_{b}|x|^{b}f(x) =\frac{d f}{d x}(x)+b\frac{f(x)-f(-x)}{x}$ is the Dunkl operator. To present the main theorem in a simpler form, we work with $D_{b}$ in this paper.  When necessary, we write $D_{b}$ as $D_{b,x}$ to emphasize that it acts on functions of $x$.
			
			\vspace{12pt}
			We define the Dunkl transform \cite{MR1199124} in this context, for $f\in|x|^{b}\mathcal{S}(\mathbb{R})$,
			$$\mathcal{F}_{b}f(\xi) := \frac{1}{2^{b+1/2}} \int_{-\infty}^{\infty}|\xi|^{b}|x|^{b}\left(\widetilde{J}_{b-\frac{1}{2}}\bigl(\xi x\bigr) -i\,\left(\frac{\xi x}{2}\right)\widetilde{J}_{b+\frac{1}{2}}\bigl(\xi x\bigr)\right)f(x)dx ,$$
			
			where $\widetilde{J}_{\nu}\left(w\right) :=\sum_{m=0}^{\infty}\frac{(-1)^{m}}{\Gamma\left(m+\nu+1\right)m!}\left(\frac{w}{2}\right)^{2m}$ is a normalized Bessel function, following the notation of \cite{MR2401813}.
			
			\vspace{12pt}
			We summarize the basic properties of $D_b$ and $\mathcal F_b$ needed in the sequel.
			
			\begin{fct}[Basic properties of $D_{b}$ and $\mathcal{F}_{b}$]\label{fct}
				\leavevmode
				\begin{enumerate}
					\item When $b=0$, $\mathcal{F}_{b}$ coincides with the Fourier transform on $\mathbb{R}$, and $D_{b}$ coincides with $\frac{d}{dx}$.
					\item $\mathcal{F}_{b}\overline{\mathcal{F}_{b}}=\overline{\mathcal{F}_{b}}\mathcal{F}_{b}=1$,\hspace{6pt} and\hspace{6pt} $\mathcal{F}_{b}^{4} =1$.
					\item $\mathcal{F}_{b}D_{b,x}=i\xi\mathcal{F}_{b}$,\hspace{6pt} and \hspace{6pt}$\mathcal{F}_{b}x = iD_{b,\xi}\mathcal{F}_{b}$.
					\item $\mathcal{F}_{b}$ extends uniquely to a unitary operator on $L^{2}(\mathbb{R})$. $D_{b}$ is an essentially skew-adjoint operator on $L^{2}(\mathbb{R})$.
					\item $D_{b}\left(|x|^{b}\mathcal{S}(\mathbb{R})\right)\subset |x|^{b}\mathcal{S}(\mathbb{R})$,\hspace{6pt} and\hspace{6pt} $\mathcal{F}_{b}\left(|x|^{b}\mathcal{S}(\mathbb{R})\right)\subset |x|^{b}\mathcal{S}(\mathbb{R})$.
				\end{enumerate}
			\end{fct}
			
			\vspace{3pt}
			\begin{proof}[Sketch of proof]
				$D_{b}^{2}-|x|^{2}$ has purely discrete spectrum (when $b=0$, it is the harmonic oscillator), and its eigenvectors are given in terms of Laguerre polynomials $L^{(\nu)}_{\ell}(t)$ as $\varphi_{2\ell}(x):= |x|^{b}e^{-\frac{x^{2}}{2}}L^{(b-\frac{1}{2})}_{\ell}(x^{2})$ and $\varphi_{2\ell+1}(x) := |x|^{b}x\,e^{-\frac{x^{2}}{2}}L^{(b+\frac{1}{2})}_{\ell}(x^{2})$\hspace{3pt} ($\ell\in \mathbb{Z}_{\ge0})$. The eigenvalue corresponding to $\varphi_{\ell}(x)$ is $-(2b+1+2\ell)$. In light of this, we define $\widetilde{\mathcal{F}_{b}} := i^{b+\frac{1}{2}}e^{\frac{\pi i}{4}(D_{b}^{2}-|x|^{2})}$. Since $\widetilde{\mathcal F_b}\varphi_\ell=i^{-\ell}\varphi_\ell$, the corresponding integral kernel is given by the eigenfunction expansion $\sum_{\ell=0}^{\infty}i^{-\ell}\frac{\varphi_{\ell}(\xi)\overline{\varphi_{\ell}(x)}}{\langle \varphi_{\ell},\varphi_{\ell}\rangle}$ (for a rigorous justification, one may use Abel summation, or equivalently a holomorphic semigroup argument; see \cite{MR2956043}). Using the Hille--Hardy formula, one obtains a closed expression for this series, which coincides with the kernel defining $\mathcal F_b$. Hence $\widetilde{\mathcal{F}_{b}}=\mathcal{F}_{b}$.\par
				
				\vspace{7pt}
				With this preparation, we see 1--5. \par 
				
				\vspace{3pt}
				
				\begin{enumerate}
					\item Since $\widetilde{J}_{-\frac{1}{2}}\left(u\right)=\frac{1}{\Gamma(1/2)}\cos(u)$ and $\frac{u}{2}\widetilde{J}_{\frac{1}{2}}\left(u\right)=\frac{1}{\Gamma(1/2)}\sin(u)$, $\mathcal{F}_{0}$ is equal to the Fourier transform. $D_{0}=\frac{d}{dx}$ follows from the definition.
					\item Since $\mathcal{F}_{b} = i^{b+\frac{1}{2}}e^{\frac{\pi i}{4}(D_{b}^{2}-|x|^{2})}$, $\mathcal{F}_{b}\varphi_{\ell}(x)=i^{-\ell}\varphi_{\ell}(x)$ holds and this shows the claim. 
					\item It follows from the computation $D_{b}\varphi_{2\ell}= -\varphi_{2\ell+1}-\varphi_{2\ell-1}$, $D_{b}\varphi_{2\ell+1}= (\ell+1)\varphi_{2\ell+2}+\left(\ell+\frac{1}{2}+b\right)\varphi_{2\ell}$, $x\varphi_{2\ell}=\varphi_{2\ell+1}-\varphi_{2\ell-1}$, $x\varphi_{2\ell+1}= -(\ell+1)\varphi_{2\ell+2}+\left(\ell+\frac{1}{2}+b\right)\varphi_{2\ell}$ and $\mathcal{F}_{b}\varphi_{\ell}(x)=i^{-\ell}\varphi_{\ell}(x)$.
					\item The unitarity of $\mathcal{F}_{b}$ follows from the representation $\mathcal{F}_{b} = i^{b+\frac{1}{2}}e^{\frac{\pi i}{4}(D_{b}^{2}-|x|^{2})}$. By item~3, $D_b$ is unitarily equivalent via $\mathcal F_b$ to multiplication by $i\xi$. Since multiplication by $\xi$ is essentially self-adjoint, $D_b$ is essentially skew-adjoint.
					\item It is equivalent to show that  $\left(|x|^{-b}\,D_{b}\,|x|^{b}\right)\mathcal{S}(\mathbb{R})\subset \mathcal{S}(\mathbb{R})$,\hspace{6pt} and\hspace{6pt} $\left(|x|^{-b} \,\mathcal{F}_{b}\,|x|^{b} \right)\mathcal{S}(\mathbb{R})\subset \mathcal{S}(\mathbb{R})$. We recall that $T_{b}f(x) = |x|^{-b}D_{b}|x|^{b}f(x) =\frac{d f}{d x}(x)+b\frac{f(x)-f(-x)}{x}$. Since the difference quotient $\frac{f(x)-f(-x)}{x}$ maps Schwartz functions to Schwartz functions, $T_{b}\,\mathcal{S}(\mathbb{R})\subset \mathcal{S}(\mathbb{R})$. Moreover $\mathcal{S}(\mathbb{R})$ is characterized as the space of functions $f\in C^{\infty}(\mathbb{R})$ such that for all $m,n\in\mathbb{N}$, $|x|^{m}T_{b}^{n}f(x)$ is bounded. This condition is invariant under $|x|^{-b} \,\mathcal{F}_{b}\,|x|^{b}$ by item~3 and hence  $\left(|x|^{-b} \,\mathcal{F}_{b}\,|x|^{b} \right)\mathcal{S}(\mathbb{R})\subset \mathcal{S}(\mathbb{R})$ holds.
				\end{enumerate}
				\end{proof}
				
				\subsection{Main Theorem}\label{1.2}
				
				From Fact~\ref{fct}, item 4, there exists a one-parameter unitary group $e^{tD_{b}}$ generated by $D_{b}$. The main result of this paper is an explicit computation of $e^{tD_{b}}$. 
				
				\vspace{12pt}
				We define the Legendre function of the second kind as
				\begin{align*}
					\MoveEqLeft \hspace{24pt}Q_{\nu}(w):=\frac{1}{2}\frac{\Gamma\left(\frac{\nu+1}{2}\right)\Gamma\left(\frac{\nu+2}{2}\right)}{\Gamma\left(\nu+\frac{3}{2}\right)}\frac{1}{w^{\nu+1}} {}_{2}F_{1}\left(\begin{matrix}
						\frac{\nu+1}{2}, \frac{\nu+2}{2} \\ \nu+\frac{3}{2} \\
					\end{matrix};\frac{1}{w^2}\right)  &\\& \hspace{36pt}=\frac{1}{2}\sum_{m=0}^{\infty}\frac{\Gamma\left(\frac{\nu+1}{2}+m\right)\Gamma\left(\frac{\nu+2}{2}+m\right)}{\Gamma\left(\nu+3/2+m\right)m!}w^{-2m-\nu-1} \hspace{12pt}(1<w)
				\end{align*}
				together with its analytic continuation.\par
				We set 
				\[\Psi_{b}\left(w\right):= bw^{b}\Bigl(Q_{b-1}\left(w\right)-Q_{b}\left(w\right)  \Bigr)\]
				and
				\[\Phi_{b}(x,y;z):= \left(\frac{x^2+y^2-z^2}{2}\right)^{-b}\Psi_{b}\left(\frac{x^2+y^2-z^2}{2xy}\right) .\]
				
				With this function, $e^{tD_{b}}$ admits a boundary value representation:
				
				\vspace{6pt}
				\phantomsection\label{mtA}
				\begin{namedthm}{Theorem A (see Theorem~\ref{mainA})}
					For $b > -\frac{1}{2}$ and $f\in |y|^{b}\mathcal{S}(\mathbb{R})$,
					\[e^{tD_{b}}f(x)= |x|^{b}\lim_{\varepsilon\rightarrow +0}\int_{-\infty}^{\infty}\frac{-1}{2\pi i}\left(\frac{\Phi_{b}\bigl(x,y;t+i\varepsilon\bigr)}{x+(t+i\varepsilon)-y}  -\frac{\Phi_{b}\bigl(x,y;t-i\varepsilon\bigr)}{x+(t-i\varepsilon)-y} \right)f(y)\,|y|^{b}dy. \]
				\end{namedthm}
				Here the boundary values are taken from the upper and lower half-planes in the variable \(z\). 
				
				In addition to the boundary value representation, one can rewrite $e^{tD_b}f(x)$ as a real-variable formula.
				
				We set
				\[K_{b}(x,y;t) := 
				\frac{-1}{2\pi i} |x|^{b}|y|^{b}\lim_{\varepsilon\rightarrow +0}\Bigl(\Phi_{b}\bigl(x,y;t+i\varepsilon\bigr)-\Phi_{b}\bigl(x,y;t-i\varepsilon\bigr)\Bigr).  \]
				
				\vspace{6pt}
				Then,
				
				\phantomsection\label{mtB}
				\begin{namedthm}{Theorem B (see Theorem~\ref{mainB})}
					For $b > -\frac{1}{2}$, $x(x+t)\neq 0$ and $f\in |y|^{b}\mathcal{S}(\mathbb{R})$,
					\begin{align*}
						\MoveEqLeft e^{tD_{b}}f(x) &\\
						\MoveEqLeft=\begin{dcases}
							f(x+t)+ \int_{\left||x|-|y|\right|<|t|} p.v._{y}\left(\frac{1}{x+t-y}\right)K_{b}(x,y;t)f(y)\,dy & \text{if}\hspace{14pt} x(x+t)>0, \\
							f(x+t)\cos(b\pi) + \int_{\left||x|-|y|\right|<|t|} p.v._{y}\left(\frac{1}{x+t-y}\right) K_{b}(x,y;t)f(y)\,dy & \text{if}\hspace{14pt} x(x+t) < 0.\\
						\end{dcases}
					\end{align*}
					
				\end{namedthm}
				Here the principal value is taken at \(y=x+t\). The behavior at \(x=0\) is treated in Remark~\ref{x=0}.
				
				Moreover, the kernel $K_b(x,y;t)$ appearing in Theorem B admits the following explicit expression.
				
				\phantomsection\label{mtC}
				\begin{namedthm}{Theorem C (see Theorem~\ref{eva})}
				\begin{align*}
					\MoveEqLeft K_{b}(x,y;t)  =b\,\mathrm{sgn}(t) \\
					\MoveEqLeft\hspace{0pt}\times\begin{dcases}
						0 & \text{if }\hspace{3pt} |t| < \bigl||x|-|y|\bigr|, \\
						\frac{1}{2}\left\{- P_{b-1}\left(\frac{x^{2}+y^{2}-t^{2}}{2|x||y|}  \right)+\mathrm{sgn}(xy)P_{b}\left(\frac{x^{2}+y^{2}-t^{2}}{2|x||y|}  \right)  \right\} & \text{if }\hspace{3pt} \bigl||x|-|y|\bigr| <|t| < |x|+|y|, \\
						-\frac{\sin(b\pi)}{\pi}\left\{Q_{b-1}\left(-\frac{x^{2}+y^{2}-t^{2}}{2|x||y|}  \right) +\mathrm{sgn}(xy)Q_{b}\left(-\frac{x^{2}+y^{2}-t^{2}}{2|x||y|}  \right) \right\} & \text{if }\hspace{3pt} |x|+|y| <|t|.
					\end{dcases} 
				\end{align*}
			\end{namedthm}
				
				Here
				\[P_{\nu}(u):={}_{2}F_{1}\left(\begin{matrix}
					-\nu,\nu+1  \\ 1 \\
				\end{matrix};\frac{1-u}{2}\right) =\sum_{m=0}^{\infty}\frac{(-\nu)_{m}(\nu+1)_{m}}{m!m!}\left(\frac{1-u}{2}\right)^{m} \hspace{12pt} (-1 < u \le 1 )\]
				
				is the Legendre function of the first kind.

				\vspace{16pt}
				
				\begin{rmk}[An analogue of the finite propagation property]
					By Main Theorem \hyperref[mtB]{B}, the value of \(e^{tD_b}f(x)\) depends only on the values of \(f(y)\) for \(\bigl||x|-|y|\bigr|\le|t|\). 
					In particular, if \(f(y)=g(y)\) for \(\bigl||x|-|y|\bigr|\le|t|\), then
					\[
					e^{tD_b}f(x)=e^{tD_b}g(x).
					\]
				\end{rmk}
				
				\begin{rmk}[Expanded form of Main Theorem B]
					Using Theorem~\hyperref[mtC]{C}, we can write $e^{tD_b}f(x)$ explicitly in terms of Legendre functions,  as follows, without using $K_b(x,y;t)$. We note that, in Case 3, the integral is understood in the sense of Cauchy's principal value at $y= x + t$.

					\begin{enumerate}
						\item When $|x| > |t| $, 
						\begin{align*}
							\MoveEqLeft e^{tD_{b}}f(x)= f(x+t) &\\
							& +\frac{b}{2} \, \int_{x-t}^{x+t} \frac{\, -P_{b-1}\left(\frac{x^2+y^2-t^2}{2|x||y|}\right) +P_{b}\left(\frac{x^2+y^2-t^2}{2|x||y|}\right) }{x+t-y}f(y)\,dy \\
							& +\frac{b}{2} \, \int_{-x-t}^{-x+t} \frac{\, -P_{b-1}\left(\frac{x^2+y^2-t^2}{2|x||y|}\right) -P_{b}\left(\frac{x^2+y^2-t^2}{2|x||y|}\right) }{x+t-y}f(y)\,dy. 
						\end{align*}
						\item When $|x| < |t| $ and $xt >0$,
						\begin{align*}
							\MoveEqLeft e^{tD_{b}}f(x)= f(x+t) &\\
							& +\frac{b}{2} \, \int_{-x+t}^{x+t} \frac{\, -P_{b-1}\left(\frac{x^2+y^2-t^2}{2|x||y|}\right) +P_{b}\left(\frac{x^2+y^2-t^2}{2|x||y|}\right) }{x+t-y}f(y)\,dy \\
							& -\frac{b\sin(b\pi)}{\pi} \, \int_{0}^{-x+t} \frac{\, Q_{b-1}\left(-\frac{x^2+y^2-t^2}{2|x||y|}\right) +Q_{b}\left(-\frac{x^2+y^2-t^2}{2|x||y|}\right) }{x+t-y}f(y)\,dy \\
							& -\frac{b\sin(b\pi)}{\pi} \, \int_{x-t}^{0} \frac{\, Q_{b-1}\left(-\frac{x^2+y^2-t^2}{2|x||y|}\right) -Q_{b}\left(-\frac{x^2+y^2-t^2}{2|x||y|}\right) }{x+t-y}f(y)\,dy \\
							& +\frac{b}{2} \, \int_{-x-t}^{x-t} \frac{\, -P_{b-1}\left(\frac{x^2+y^2-t^2}{2|x||y|}\right) -P_{b}\left(\frac{x^2+y^2-t^2}{2|x||y|}\right) }{x+t-y}f(y)\,dy. 
						\end{align*}
						\item When $|x| < |t| $ and $xt < 0$,
						\begin{align*}
							\MoveEqLeft e^{tD_{b}}f(x)= f(x+t)\cos(b\pi) &\\
							& +\frac{b}{2} \, \int_{x+t}^{-x+t} \frac{\, -P_{b-1}\left(\frac{x^2+y^2-t^2}{2|x||y|}\right) -P_{b}\left(\frac{x^2+y^2-t^2}{2|x||y|}\right) }{x+t-y}f(y)\,dy \\
							& -\frac{b\sin(b\pi)}{\pi} \, \int_{0}^{x+t} \frac{\, Q_{b-1}\left(-\frac{x^2+y^2-t^2}{2|x||y|}\right) -Q_{b}\left(-\frac{x^2+y^2-t^2}{2|x||y|}\right) }{x+t-y}f(y)\,dy \\
							& -\frac{b\sin(b\pi)}{\pi} \, \int_{-x-t}^{0} \frac{\, Q_{b-1}\left(-\frac{x^2+y^2-t^2}{2|x||y|}\right) +Q_{b}\left(-\frac{x^2+y^2-t^2}{2|x||y|}\right) }{x+t-y}f(y)\,dy \\
							& +\frac{b}{2} \, \int_{x-t}^{-x-t} \frac{\, -P_{b-1}\left(\frac{x^2+y^2-t^2}{2|x||y|}\right) +P_{b}\left(\frac{x^2+y^2-t^2}{2|x||y|}\right) }{x+t-y}f(y)\,dy. 
						\end{align*}
						
					\end{enumerate}
					
					\vspace{12pt}
					
				\end{rmk}

			\section{Proof of the Main Theorem}
			
			\subsection{Outline of the Proof}
			
			We will show Main Theorems~\hyperref[mtA]{A}, \hyperref[mtB]{B}, and  \hyperref[mtC]{C} in Subsections~\ref{4.1}--\ref{4.6}. We now outline the argument.
			
			Let $B_{b}(\xi,x):=\frac{1}{2^{b+1/2}}|\xi|^{b}|x|^{b}\left(\widetilde{J}_{b-\frac{1}{2}}\bigl(\xi x\bigr) -i\,\left(\frac{\xi x}{2}\right)\widetilde{J}_{b+\frac{1}{2}}\bigl(\xi x\bigr)\right)$. 
			Since $e^{tD_{b}}= \mathcal{F}_{b}\,e^{-itx}\, \mathcal{F}_{b}^{-1}$,
			
			\begin{align*}
				\MoveEqLeft e^{tD_{b}}f(x) = \int_{-\infty}^{\infty} B_{b}(x,\xi)
				e^{-it\xi}
				\left(\int_{-\infty}^{\infty} \overline{B_{b}(\xi,y)}f(y)\,dy\right)d\xi.& 
			\end{align*}
			
			In Subsection~\ref{4.1}, we decompose it into its positive- and negative-frequency parts as 
			\begin{align}\label{decomp2}
				\MoveEqLeft e^{tD_{b}}f(x) = \int_{0}^{\infty} B_{b}(x,\xi)
				e^{-it\xi}
				\left(\int_{-\infty}^{\infty} \overline{B_{b}(\xi,y)}f(y)dy\right)d\xi \nonumber\\
				&+ \int_{-\infty}^{0} B_{b}(x,\xi)
				e^{-it\xi}
				\left(\int_{-\infty}^{\infty} \overline{B_{b}(\xi,y)}f(y)dy\right)d\xi& 
			\end{align}
			and extend each term to the lower and the upper half-planes, respectively. Then, in the interior of each plane, we may apply Fubini's theorem. 
			
			\vspace{12pt}
			In Subsection~\ref{4.2}, we evaluate $ \int_{0}^{\infty} B_{b}(x,\xi)\overline{B_{b}(y,\xi)}
			e^{-iz\xi}d\xi \hspace{12pt} (\mathrm{Im}(z) <0 )$. \par 
			We set 
			$\Psi_{b}\left(w\right):=bw^{b}\Bigl(Q_{b-1}\left(w\right)- Q_{b}\left(w\right) \Bigr).$
			Then for $b >-\frac{1}{2}$ and $ \mathrm{Im}(z) <0$,
			\begin{equation}\label{int3}
				\int_{0}^{\infty} B_{b}(x,\xi)\overline{B_{b}(y,\xi)}
				e^{-iz\xi}d\xi  = \frac{1}{2\pi i}\frac{1}{x+z-y}|x|^{b}|y|^{b}\left(\frac{x^2+y^2-z^2}{2}\right)^{-b} \Psi_{b}\left(\frac{x^2+y^2-z^2}{2xy}\right).
			\end{equation}
			
			\vspace{12pt}
			In Subsection~\ref{4.3}, by Formulas~\eqref{decomp2} and \eqref{int3}, we obtain the boundary value representation. \\ We set 
			$\Phi_{b}(x,y;z):= \left(\frac{x^2+y^2-z^2}{2}\right)^{-b}\Psi_{b}\left(\frac{x^2+y^2-z^2}{2xy}\right) .$
			Then,
			 
			\[e^{tD_{b}}f(x)= |x|^{b}\lim_{\varepsilon\rightarrow +0}\int_{-\infty}^{\infty}\frac{-1}{2\pi i}\left(\frac{\Phi_{b}\bigl(x,y;t+i\varepsilon\bigr)}{x+(t+i\varepsilon)-y}  -\frac{\Phi_{b}\bigl(x,y;t-i\varepsilon\bigr)}{x+(t-i\varepsilon)-y} \right)f(y)\,|y|^{b}dy. \] 
			This is our Main Theorem~\hyperref[mtA]{A}.
			By the identities, $\frac{1}{x+i0}=p.v.\left(\frac{1}{x}\right)-i\pi\delta(x)$
			and $\frac{1}{x-i0}=p.v.\left(\frac{1}{x}\right)+i\pi\delta(x)$ we obtain 
			\begin{align}
				\MoveEqLeft e^{tD_{b}}f(x) = |x|^{b}\int_{-\infty}^{\infty}\delta\bigl(x+t-y\bigr) \frac{1}{2}\Bigl(\Phi_{b}\bigl(x,y;t+i0\bigr)+\Phi_{b}\bigl(x,y;t-i0\bigr)\Bigr)f(y)\,|y|^{b}dy\nonumber \\
				\MoveEqLeft \hspace{43pt}+|x|^{b}\int_{-\infty}^{\infty}p.v._{y}\left(\frac{1}{x+t-y}\right)
				\frac{-1}{2\pi i} \Bigl(\Phi_{b}\bigl(x,y;t+i0\bigr)-\Phi_{b}\bigl(x,y;t-i0\bigr)\Bigr)f(y)\,|y|^{b}dy .\label{dist} 
			\end{align}
			
			\vspace{3pt}
			In Subsection~\ref{4.4}, we evaluate $\lim_{y\to x+t}\Bigl(\Phi_{b}\bigl(x,y;t+i0\bigr)+\Phi_{b}\bigl(x,y;t-i0\bigr)\Bigr)$ which computes the contribution of the $\delta$-term in Formula~\eqref{dist}.
			
			\vspace{12pt}
			In Subsection~\ref{4.5}, we complete the proof of our Main Theorem~\hyperref[mtB]{B}.
			
			\vspace{12pt}
			In Subsection~\ref{4.6}, we prove Main Theorem~\hyperref[mtC]{C} by investigating the boundary behavior of $Q_{\nu}(w)$ along the real axis.
			This evaluates 
			$K_{b}(x,y;t):=\frac{-1}{2\pi i} |x|^{b}|y|^{b}\Bigl(\Phi_{b}\bigl(x,y;t+i0\bigr)-\Phi_{b}\bigl(x,y;t-i0\bigr)\Bigr)$ and hence computes the contribution of the Cauchy principal value term in Formula~\eqref{dist}.
			
			\vspace{12pt}
			
			\subsection{Positive and Negative Frequency Decomposition}\label{4.1}
			
			Since \(e^{tD_{b}}= \mathcal{F}_{b}\,e^{-itx}\, \mathcal{F}_{b}^{-1}\),
			\begin{align*}
				e^{tD_{b}}f(x)
				=
				\int_{-\infty}^{\infty} B_{b}(x,\xi)e^{-it\xi}
				\left(\int_{-\infty}^{\infty} \overline{B_{b}(\xi,y)}f(y)\,dy\right)d\xi .
			\end{align*}
			Here
			$B_{b}(\xi,x):=\frac{1}{2^{b+1/2}}|\xi|^{b}|x|^{b}
			\left(\widetilde{J}_{b-\frac{1}{2}}(\xi x)
			-i\left(\frac{\xi x}{2}\right)\widetilde{J}_{b+\frac{1}{2}}(\xi x)\right)$.
			
			\vspace{12pt}
			
			We decompose the above formula into its positive- and negative-frequency parts. We set
			\begin{align*}
				I_{+}(t,x;f)
				&:=
				\int_{0}^{\infty} B_{b}(x,\xi)e^{-it\xi}
				\left(\int_{-\infty}^{\infty} \overline{B_{b}(\xi,y)}f(y)\,dy\right)d\xi,\\
				I_{-}(t,x;f)
				&:=
				\int_{-\infty}^{0} B_{b}(x,\xi)e^{-it\xi}
				\left(\int_{-\infty}^{\infty} \overline{B_{b}(\xi,y)}f(y)\,dy\right)d\xi .
			\end{align*}
			Then
			\[
			e^{tD_{b}}f(x)=I_{+}(t,x;f)+I_{-}(t,x;f).
			\]
			
			Using \(B_{b}(x,-\xi)=\overline{B_{b}(x,\xi)}\) and \(\overline{B_b(-\xi,y)}=B_b(\xi,y)\), we have
			
			\[I_{-}(t,x;f)=\overline{I_{+}(t,x;\overline{f})}.\]
			
			We extend the real parameter \(t\) to a complex variable \(z\) in the lower and upper half-planes by
			\begin{align*}
				I_{+}(z,x;f)
				&:=
				\int_{0}^{\infty} B_{b}(x,\xi)e^{-iz\xi}
				\left(\int_{-\infty}^{\infty} \overline{B_{b}(\xi,y)}f(y)\,dy\right)d\xi
				\qquad (\mathrm{Im}(z)\le 0),\\
				I_{-}(z,x;f)
				&:=
				\int_{-\infty}^{0} B_{b}(x,\xi)e^{-iz\xi}
				\left(\int_{-\infty}^{\infty} \overline{B_{b}(\xi,y)}f(y)\,dy\right)d\xi
				\qquad (\mathrm{Im}(z)\ge 0).
			\end{align*}
			
			When \(\mathrm{Im}(z)<0\), we may apply Fubini's theorem to \(I_{+}(z,x;f)\), and obtain
			\begin{align*}
				I_{+}(z,x;f)
				=
				\int_{-\infty}^{\infty}
				\left(
				\int_{0}^{\infty} B_{b}(x,\xi)\overline{B_{b}(y,\xi)}e^{-iz\xi}\,d\xi
				\right)
				f(y)\,dy .
			\end{align*}

			\vspace{12pt}
			\subsection{An Integral Formula Involving Bessel Functions}\label{4.2}
			In this subsection, we evaluate $ \int_{0}^{\infty} B_{b}(x,\xi)\overline{B_{b}(y,\xi)}
			e^{-iz\xi}d\xi $. 
			
				We set 
				\begin{align*}
					\MoveEqLeft \Psi_{b}\left(w\right):= bw^{b}\Bigl(Q_{b-1}\left(w\right)- Q_{b}\left(w\right) \Bigr) &\\  \MoveEqLeft =\frac{b}{2}\left(\sum_{m=0}^{\infty}\frac{\Gamma\left(\frac{b}{2}+m\right)\Gamma\left(\frac{b+1}{2}+m\right)}{\Gamma\left(b+1/2+m\right)m!}w^{-2m}-\sum_{m=0}^{\infty}\frac{\Gamma\left(\frac{b+1}{2}+m\right)\Gamma\left(\frac{b+2}{2}+m\right)}{\Gamma\left(b+3/2+m\right)m!}w^{-2m-1} \right)\hspace{12pt}(1<w)
				\end{align*}
				
				together with its analytic continuation. Here $Q_{\nu}(w)$ is the Legendre function of the second kind; see the beginning of Subsection~\ref{1.2} for its definition.
				
				\begin{rmk}[Properties of $\Psi_{b}\left(w\right)$]\label{Psi}
					We note some elementary properties of $\Psi_{b}\left(w\right)$.
					\begin{enumerate}
						\item $\Psi_{0}\left(w\right) =1$. 
						\item $\Psi_{b}\left(w\right)$ is single-valued on the domain $|w|>1$. 
						\item $\lim_{w\to \infty}\Psi_{b}\left(w\right)=\frac{\Gamma\left(\frac{b}{2}+1\right)\Gamma\left(\frac{b+1}{2}\right)}{\Gamma\left(b+\frac{1}{2}\right)}$.
					\end{enumerate}
					
					We also note that
					\[
					\lim_{w\to1}\Psi_b(w)=1,
					\]
					which will be shown in Lemma~\ref{del}.
					
				\end{rmk}
			
			\hspace{12pt}
			\begin{prp}\label{half}
				For $b >-\frac{1}{2}$ and $ \mathrm{Im}(z) <0$,
				\begin{align*}
					\MoveEqLeft \int_{0}^{\infty} B_{b}(x,\xi)\overline{B_{b}(y,\xi)}
					e^{-iz\xi}d\xi  \\
					&= \frac{1}{2\pi i}\frac{1}{x+z-y} \,b\left\{Q_{b-1}\left(\frac{x^2+y^2-z^2}{2|x||y|}\right)-\mathrm{sgn}(xy)Q_{b}\left(\frac{x^2+y^2-z^2}{2|x||y|}\right)\right\}\\
					& = \frac{1}{2\pi i}\frac{1}{x+z-y} |x|^{b}|y|^{b}\left(\frac{x^2+y^2-z^2}{2}\right)^{-b} \Psi_{b}\left(\frac{x^2+y^2-z^2}{2xy}\right).
				\end{align*}
			\end{prp}
			
			\begin{proof}
			
			Since 
			\begin{align*}
				\MoveEqLeft \int_{0}^{\infty} B_{b}(x,\xi)\overline{B_{b}(y,\xi)}
				e^{-iz\xi}d\xi  = \frac{1}{2^{2b+1}}|x|^{b}|y|^{b}  \\ 
				\MoveEqLeft\hspace{12pt} \times\int_{0}^{\infty}\left(\widetilde{J}_{b-\frac{1}{2}}\bigl(x\xi\bigr) -i\,\left(\frac{\xi x}{2}\right) \,\widetilde{J}_{b+\frac{1}{2}}\bigl(x \xi\bigr)\right)
				\left(\widetilde{J}_{b-\frac{1}{2}}\bigl(\xi y\bigr) +i\,\left(\frac{\xi y}{2}\right) \,\widetilde{J}_{b+\frac{1}{2}}\bigl(\xi y\bigr)\right) e^{-iz\xi}\,\xi^{2b}d\xi, 
			\end{align*}
			
			we compute it by the following Fact and Lemmas.
			\begin{fct}
				(See \cite[6.612,item~3]{MR3307944} and \cite[Section 13.22]{MR10746}, for references.)\par Suppose $\mathrm{Re}(\gamma\pm i\alpha \pm i\beta)>0$, and $\nu > 0$. Then,
				\begin{align*}
					\MoveEqLeft \frac{\alpha^{\nu}\beta^{\nu}}{2^{2\nu+1}}\int_{0}^{\infty}\,\,\widetilde{J}_{\nu-\frac{1}{2}}(\alpha \xi)\widetilde{J}_{\nu-\frac{1}{2}}(\beta \xi)e^{-\gamma \xi}\,\xi^{2\nu-1}d\xi = \frac{1}{4\pi}Q_{\nu-1}\left(\frac{\alpha^2+\beta^2+\gamma^2}{2\alpha\beta}\right).
				\end{align*}
			\end{fct}
			
			\begin{lmm}
				\[D_{b,y} \biggl(|y|^{b}\,\widetilde{J}_{b-1/2}\bigl(\xi y\bigr)\biggr)=-\xi\,\biggl(|y|^{b} \left(\frac{\xi y}{2}\right)\widetilde{J}_{b+1/2}\bigl(\xi y\bigr)\biggr)\]
				\[ D_{b,y} \biggl(|y|^{b} \left(\frac{\xi y}{2}\right)\widetilde{J}_{b+1/2}\bigl(\xi y\bigr)\biggr) = \xi\,\biggl(|y|^{b}\,\widetilde{J}_{b-1/2}\bigl(\xi y\bigr)\biggr).\]
			\end{lmm}
			\begin{proof}
				By Fact~\ref{fct}, item~3, $D_{b,y}B_{b}(\xi,y)=-i\xi B_{b}(\xi,y)$. This shows the claim. 
				
				\vspace{3pt}
				(We note that $B_{b}(\xi,y)=2^{-(b+1/2)}|\xi|^{b}|y|^{b}\left(\widetilde{J}_{b-\frac{1}{2}}\bigl(\xi y\bigr) -i\,\left(\xi y/2\right)\widetilde{J}_{b+\frac{1}{2}}\bigl(\xi y\bigr)\right)$.)
			\end{proof}
			\vspace{12pt}
			
			\begin{lmm}\label{diff}
			\begin{align*}
				\MoveEqLeft\frac{\partial}{\partial z} \left(Q_{b-1}\left(\frac{x^2+y^2-z^2}{2|x||y|}\right)\right) \nonumber&\\
				& =2b\,\frac{z(x^2+y^2-z^2)\,Q_{b-1}\left(\frac{x^2+y^2-z^2}{2|x||y|}\right)-2|x||y|z \,Q_{b}\left(\frac{x^2+y^2-z^2}{2|x||y|}\right)}{(x+y+z)(x+y-z)(x-y+z)(x-y-z)}
			\end{align*}
			\begin{align*}
				\MoveEqLeft\frac{\partial}{\partial z} \left(Q_{b}\left(\frac{x^2+y^2-z^2}{2|x||y|}\right)\right) \nonumber&\\ & =2b\,\frac{2|x||y|z\,Q_{b-1}\left(\frac{x^2+y^2-z^2}{2|x||y|}\right)-z(x^2+y^2-z^2)\,Q_{b}\left(\frac{x^2+y^2-z^2}{2|x||y|}\right)}{(x+y+z)(x+y-z)(x-y+z)(x-y-z)}
			\end{align*}
			\begin{align*}
				\MoveEqLeft D_{b,y} \left(Q_{b-1}\left(\frac{x^2+y^2-z^2}{2|x||y|}\right)\right) \nonumber&\\
				\MoveEqLeft\hspace{12pt} = 2b\,\mathrm{sgn}(y)\, \frac{|y|(x^2-y^2+z^2)\,Q_{b-1}\left(\frac{x^2+y^2-z^2}{2|x||y|}\right)+|x|(-x^2+y^2+z^2)\,Q_{b}\left(\frac{x^2+y^2-z^2}{2|x||y|}\right)}{(x+y+z)(x+y-z)(x-y+z)(x-y-z)} 
			\end{align*}
			\begin{align*}
				\MoveEqLeft D_{b,y} \left(\mathrm{sgn}(y) Q_{b}\left(\frac{x^2+y^2-z^2}{2|x||y|}\right)\right) \nonumber&\\
				\MoveEqLeft\hspace{12pt}  = 2b\,  \frac{-|x|(-x^2+y^2+z^2)\,Q_{b-1}\left(\frac{x^2+y^2-z^2}{2|x||y|}\right)-|y|(x^2-y^2+z^2)\,Q_{b}\left(\frac{x^2+y^2-z^2}{2|x||y|}\right)}{(x+y+z)(x+y-z)(x-y+z)(x-y-z)}.
			\end{align*}
		\end{lmm}
		\begin{proof}
			It follows from the computation using
			
			\[ (w^2-1)\frac{dQ_{b-1}(w)}{dw} = b\left(-wQ_{b-1}(w)+Q_{b}(w)\right) \]
			\[ (w^2-1)\frac{dQ_{b}(w)}{dw} = b\left(-Q_{b-1}(w)+wQ_{b}(w)\right)\]
			
			(see \cite[8.832, item~3 and 8.732, item~2]{MR3307944}, for a reference), 
			
			\[\left( \frac{x^2+y^2-z^2}{2|x||y|} \right)^{2}-1 
				= \frac{(x+y+z)(x+y-z)(x-y+z)(x-y-z)}{4x^2y^2},\]
			\[\frac{\partial}{\partial y}\frac{x^2+y^2-z^2}{2|x||y|} = \frac{1}{y}\frac{-x^2+y^2+z^2}{2|x||y|}, \quad\text{ and}\qquad \frac{\partial}{\partial z}\frac{x^2+y^2-z^2}{2|x||y|} = -\frac{z}{|x||y|}. \]
		\end{proof}
		
		By Lemma~\ref{diff},
		
		\begin{align*}
			\MoveEqLeft\frac{\partial}{\partial z} \Biggl(Q_{b-1}\left(\frac{x^2+y^2-z^2}{2|x||y|}\right)+\mathrm{sgn}(xy)Q_{b}\left(\frac{x^2+y^2-z^2}{2|x||y|}\right)\Biggr) \nonumber&\\
			& =2b\,\frac{z}{(x- y+z)(x- y-z)}\Biggl(Q_{b-1}\left(\frac{x^2+y^2-z^2}{2|x||y|}\right)- \mathrm{sgn}(xy)Q_{b}\left(\frac{x^2+y^2-z^2}{2|x||y|}\right)\Biggr) 
		\end{align*}
		
		\begin{align*}
			\MoveEqLeft D_{b,y} \Biggl(Q_{b-1}\left(\frac{x^2+y^2-z^2}{2|x||y|}\right)+\mathrm{sgn}(xy)Q_{b}\left(\frac{x^2+y^2-z^2}{2|x||y|}\right)\Biggr) \nonumber&\\
			\MoveEqLeft\hspace{12pt} = 2b\,\,\frac{x-y}{(x-y+z)(x-y-z)}\Biggl(Q_{b-1}\left(\frac{x^2+y^2-z^2}{2|x||y|}\right)-\mathrm{sgn}(xy)Q_{b}\left(\frac{x^2+y^2-z^2}{2|x||y|}\right)\Biggr) 
		\end{align*}
		
		\vspace{12pt}
		
		First we assume $b>0$. Then
		
		\begin{align*}
			\MoveEqLeft \int_{0}^{\infty} B_{b}(x,\xi)\overline{B_{b}(y,\xi)}
			e^{-iz\xi}d\xi  = \frac{1}{2^{2b+1}}|x|^{b}|y|^{b} &\\
			\MoveEqLeft\hspace{12pt} \times\int_{0}^{\infty}\left(\widetilde{J}_{b-\frac{1}{2}}\bigl(x\xi\bigr) -i\,\left(\frac{\xi x}{2}\right) \,\widetilde{J}_{b+\frac{1}{2}}\bigl(x \xi\bigr)\right)
			\left(\widetilde{J}_{b-\frac{1}{2}}\bigl(\xi y\bigr) +i\,\left(\frac{\xi y}{2}\right) \,\widetilde{J}_{b+\frac{1}{2}}\bigl(\xi y\bigr)\right) e^{-iz\xi}\,\xi^{2b}d\xi \\[15pt]
			\MoveEqLeft=\frac{1}{4\pi}\left\{i\frac{\partial}{\partial z}\left(Q_{b-1}\left(\frac{x^2+y^2-z^2}{2|x||y|}\right)\right)-iD_{b,y}\left(Q_{b-1}\left(\frac{x^2+y^2-z^2}{2|x||y|}\right)\right)\right. &\\
			\MoveEqLeft\left. \hspace{12pt}-iD_{b,y}\left(\mathrm{sgn}(xy)\,Q_{b}\left(\frac{x^2+y^2-z^2}{2|x||y|}\right)\right) +i\frac{\partial}{\partial z}\left( \mathrm{sgn}(xy)\,Q_{b}\left(\frac{x^2+y^2-z^2}{2|x||y|}\right)\right) \right\} & \\[15pt]
			\MoveEqLeft =\frac{1}{2\pi i}\frac{1}{x+z-y}\, b\left\{Q_{b-1}\left(\frac{x^2+y^2-z^2}{2|x||y|}\right)-\mathrm{sgn}(xy)Q_{b}\left(\frac{x^2+y^2-z^2}{2|x||y|}\right)\right\}&\\[6pt]
			\MoveEqLeft =\frac{1}{2\pi i}\frac{1}{x+z-y} |x|^{b}|y|^{b}\left(\frac{x^2+y^2-z^2}{2}\right)^{-b} \Psi_{b}\left(\frac{x^2+y^2-z^2}{2xy}\right).&
		\end{align*}
		
		By analytic continuation in \(b\), the claim follows for all \(b>-\frac12\).
		\end{proof}

		\vspace{12pt}
		\subsection{Boundary Value Representation}\label{4.3}
		
		We set
		\[
		\Phi_{b}(x,y;z):= \left(\frac{x^2+y^2-z^2}{2}\right)^{-b}\Psi_{b}\left(\frac{x^2+y^2-z^2}{2xy}\right) .
		\]
		
		Equivalently,
		\[
		|x|^b |y|^b \Phi_b(x,y;z)
		=
		b\left\{
		Q_{b-1}\left(\frac{x^2+y^2-z^2}{2|x||y|}\right)
		-\mathrm{sgn}(xy)\,
		Q_b\left(\frac{x^2+y^2-z^2}{2|x||y|}\right)
		\right\}.
		\]
		
		\vspace{12pt}
		
		\begin{thm}[Boundary Value Representation]\label{mainA}
		For $b > -\frac{1}{2}$ and $f\in |y|^{b}\mathcal{S}(\mathbb{R})$,
		\[e^{tD_{b}}f(x)= |x|^{b}\lim_{\varepsilon\rightarrow +0}\int_{-\infty}^{\infty}\frac{-1}{2\pi i}\left(\frac{\Phi_{b}\bigl(x,y;t+i\varepsilon\bigr)}{x+(t+i\varepsilon)-y}  -\frac{\Phi_{b}\bigl(x,y;t-i\varepsilon\bigr)}{x+(t-i\varepsilon)-y} \right)f(y)\,|y|^{b}dy \]
		\end{thm}
		\begin{proof}
			By Proposition~\ref{half}, we obtain an explicit formula for \(I_{+}(z,x;f)\) in the lower half-plane, where \(I_{+}(z,x;f)\) is the integral defined in Subsection~\ref{4.1}. Passing to the boundary value \(z\to t-i0\) and using
			$I_{-}(t,x;f)=\overline{I_{+}(t,x;\overline{f})}$,
			we obtain the claim.
		\end{proof}
	 	
	 	\vspace{12pt}
	 	
	 	\begin{rmk}[The case when $b=0$]
	 		Using $\Psi_{0}\left(w\right) =1$ together with Theorem~\ref{mainA}, when $b=0$,
	 		\[ e^{tD_{b}}f(x) =f(x+t).\]
	 	\end{rmk}
	 	
	 	\begin{rmk}[Behavior at $x=0$]\label{x=0}
	 		
	 		Suppose $-\frac{1}{2} <b<0$. For $t\neq0$ and $f\in |y|^{b}\mathcal{S}(\mathbb{R})$,  
	 		\begin{align*}
	 			\MoveEqLeft\lim_{x\to +0} |x|^{-b}e^{tD_{b}}f(x)\\
	 			\MoveEqLeft \hspace{12pt}=\frac{\Gamma(\frac{b}{2}+1)\Gamma(\frac{b+1}{2})}{\Gamma(b+\frac{1}{2})} \int_{-\infty}^{\infty}\frac{-1}{2\pi i}\left(\frac{\left(y^2-(t+i0)^{2}\right)^{-b}}{t+i0-y}  -\frac{\left(y^2-(t-i0)^{2}\right)^{-b}}{t-i0-y} \right)f(y)\,|y|^{b}dy & \\
	 			\MoveEqLeft\hspace{12pt}= -\frac{\Gamma(\frac{b}{2}+1)\Gamma(\frac{b+1}{2})}{\Gamma(b+\frac{1}{2})}\frac{\sin(b\pi)}{\pi}\int_{-|t|}^{|t|} \left(t-y\right)^{-b-1}\left(t+y\right)^{-b} f(y)\,|y|^{b}dy 
	 		\end{align*}
	 		
	 		By analytic continuation of $-\frac{\sin(b\pi)}{\pi}\left(t-y\right)^{-b-1}\left(t+y\right)^{-b} |y|^{2b}$ as a distribution, the formula extends to $b >-\frac{1}{2}$.
	 	\end{rmk}
		\vspace{24pt}
		
		Using the identities
		\[ \frac{1}{x+i0}=p.v.\left(\frac{1}{x}\right)-i\pi\delta(x)\]
		\[ \frac{1}{x-i0}=p.v.\left(\frac{1}{x}\right)+i\pi\delta(x)\]
		
		we obtain
		\[\frac{1}{x+t-y+i0} =p.v._{y}\left(\frac{1}{x+t-y}\right)-i\pi\delta\bigl(x+t-y\bigr),\]
		\[\frac{1}{x+t-y-i0} =p.v._{y}\left(\frac{1}{x+t-y}\right)+i\pi\delta\bigl(x+t-y\bigr).\]
		
		\vspace{6pt}
		Hence, formally, Theorem~\ref{mainA} leads to the following expression.
		\begin{align}
			\MoveEqLeft e^{tD_{b}}f(x) \nonumber\\
			\MoveEqLeft = |x|^{b}\int_{-\infty}^{\infty}\delta\bigl(x+t-y\bigr) \frac{1}{2}\Bigl(\Phi_{b}\bigl(x,y;t+i0\bigr)+\Phi_{b}\bigl(x,y;t-i0\bigr)\Bigr)f(y)\,|y|^{b}dy \nonumber\\
			\MoveEqLeft \hspace{2pt}+|x|^{b}\int_{-\infty}^{\infty}p.v._{y}\left(\frac{1}{x+t-y}\right)
			\frac{-1}{2\pi i} \Bigl(\Phi_{b}\bigl(x,y;t+i0\bigr)-\Phi_{b}\bigl(x,y;t-i0\bigr)\Bigr)f(y)\,|y|^{b}dy \label{dis}.
		\end{align}
		
		This will be justified in Proposition~\ref{delpri}.
		\vspace{12pt}
		
		\subsection{Computation of the Contribution from the $\delta$-Term}\label{4.4}
		
		To compute the $\delta$-term in \eqref{dis}, which will be justified in Proposition~\ref{delpri}, we show the following proposition.
		\begin{prp}\label{cos}
			\begin{align*}
				\MoveEqLeft \lim_{y\to x+t}\frac{1}{2}\Bigl(\Phi_{b}\bigl(x,y;t+i0\bigr)+\Phi_{b}\bigl(x,y;t-i0\bigr)\Bigr)= \frac{1}{2}\left\{\Bigl(x(x+t)+i0\Bigr)^{-b} +\Bigl(x(x+t)-i0\Bigr)^{-b}\right\} \\
				& =\begin{cases}
					|x|^{-b}|x+t|^{-b} & \hspace{14pt}\text{if} \hspace{14pt}x(x+t) >0 \\
					|x|^{-b}|x+t|^{-b}\cos(b\pi) & \hspace{14pt}\text{if} \hspace{14pt} x(x+t) <0 \\
				\end{cases}
			\end{align*}
		\end{prp}
			
			\begin{proof}
				We recall that
				
				\[\Phi_{b}(x,y;z)= \left(\frac{x^2+y^2-z^2}{2}\right)^{-b}\Psi_{b}\left(\frac{x^2+y^2-z^2}{2xy}\right) .\]
				
				Since  
				$x^2+(x+t)^2-t^2= 2x(x+t)$,
				it suffices to show the following Lemma.
				
				\begin{lmm}\label{del}
					\[\lim_{w\to 1}\Psi_{b}\bigl(w\bigr)=1\]
				\end{lmm}
				
			\begin{proof}
				We recall that 
				\[\Psi_{b}\left(w\right)= bw^{b}\Bigl( Q_{b-1}\left(w\right)-Q_{b}\left(w\right) \Bigr) 
				\]
				
				\[Q_{\nu}(w)=\frac{\Gamma\left(\frac{\nu+1}{2}\right)\Gamma\left(\frac{\nu+2}{2}\right)}{2\Gamma(\nu+3/2)}w^{-(\nu+1)} {}_{2}F_{1}\left(\begin{matrix}
					\frac{\nu+1}{2}, \frac{\nu+2}{2} \\ \nu+\frac{3}{2} \\
				\end{matrix};\frac{1}{w^2}\right). \]

				\vspace{12pt}
				We use the following connection formula for the Gauss hypergeometric function in the case $\gamma=\alpha+\beta$; a proof is given in Appendix~\ref{App}:
				 
				\begin{align*}
					\MoveEqLeft {}_{2}F_{1}\left(\begin{matrix}
						\alpha, \beta \\ \gamma \\
					\end{matrix}; w\right) = \frac{\Gamma(\gamma)}{\Gamma(\alpha)\Gamma(\beta)}\left\{-\log\left(1-w\right) {}_{2}F_{1}\left(\begin{matrix}
						\alpha, \beta \\ 1 \\
					\end{matrix};1-w\right) \right.&\\
					&\left.- \sum_{m=0}^{\infty} \frac{(\alpha)_{m}(\beta)_{m}}{m!} \bigl(\psi(\alpha+m)+\psi(\beta+m)-2\psi(1+m) \bigr)\frac{(1-w)^{m}}{m!}\right\}.
				\end{align*}
				Here $\psi(w):=\frac{\Gamma'(w)}{\Gamma(w)}$ is the digamma function.
				
				Hence, 
				\begin{align*}
					\MoveEqLeft \lim_{z\to 1} \Psi_{b}\left(z\right) & \\
					\MoveEqLeft = -\frac{b}{2}\left(\psi\left(\frac{b+1}{2}\right)+\psi\left(\frac{b}{2}\right)-2\psi(1) \right) +
					\frac{b}{2}\left(\psi\left(\frac{b+2}{2}\right)+\psi\left(\frac{b+1}{2}\right)-2\psi(1)\right) \\
					\MoveEqLeft  = \frac{b}{2}\left(\psi\left(\frac{b}{2}+1\right)-\psi\left(\frac{b}{2}\right) \right) = 1.
				\end{align*}

				This shows the claim.
				
			\end{proof}
			
	\end{proof}

		\subsection{Real-Variable Formula}\label{4.5}
			In this subsection, we prove Proposition~\ref{delpri}, thereby justifying formula~\eqref{dis}, and then deduce Main Theorem~\hyperref[mtB]{B}.
			
			\vspace{12pt}
			First, we prepare the following Lemma.
			\begin{lmm}\label{pvl}
				Let $\{G_\varepsilon^\pm\}_{\varepsilon >0}$ be measurable functions on $\mathbb R$ satisfying the following conditions.
				
				\begin{enumerate}
					\item There exist measurable functions $G_0^\pm$ such that
					\[
					\lim_{\varepsilon\to+0}G_\varepsilon^\pm(x)=G_0^\pm(x)
					\qquad \text{for a.e. }x\in\mathbb R,
					\]
					\item 
					\[
					\lim_{\varepsilon\to+0}G_\varepsilon^\pm(0)=G_0^\pm(0)
					\]
					and $G_0^\pm(0)$ is finite.
					\item There exists $\varepsilon_{0},\delta >0$, $h_{1}\in L^1(-\delta,\delta)$ and $h_{2}\in L^1\bigl(\mathbb{R}\backslash(-\delta,\delta)\bigr)$such that
					\[
					\left|\frac{G_\varepsilon^\pm(x)-G_\varepsilon^\pm(0)}{x\pm i\varepsilon}\right| \le h_{1}(x) \qquad (0 < \varepsilon < \varepsilon_{0},\quad x\in(-\delta,\delta)).
					\]
					\[
					\left|\frac{G_\varepsilon^\pm(x)}{x\pm i\varepsilon}\right| \le h_{2}(x) \qquad (0 < \varepsilon < \varepsilon_{0}, \quad x\in \mathbb{R}\backslash(-\delta,\delta)).
					\]
				\end{enumerate}
				
				Then
				\begin{align*}
					\lim_{\varepsilon\to+0}
					\int_{-\infty}^{\infty}
					\frac{-1}{2\pi i}\left(
					\frac{G_\varepsilon^+(x)}{x+i\varepsilon}
					-
					\frac{G_\varepsilon^-(x)}{x-i\varepsilon}
					\right)\,dx =
					\frac{G_0^{+}(0)+G_{0}^{-}(0)}{2}
					+
					\frac{-1}{2\pi i} \,p.v.\int_{-\infty}^{\infty}\,
					\,\!\frac{G_0^+(x)-G_0^-(x)}{x}
					\,dx.
				\end{align*}
			\end{lmm}
			
			\vspace{12pt}
			\begin{proof}
				We first decompose as
				\begin{align*}
					&\lim_{\varepsilon\to+0}\int_{-\infty}^{\infty}\frac{-1}{2\pi i}
					\left(
					\frac{G_\varepsilon^+(x)}{x+i\varepsilon}
					-
					\frac{G_\varepsilon^-(x)}{x-i\varepsilon}
					\right)\,dx \\
					&\qquad=
					\lim_{\varepsilon\to+0}\int_{|x| < \delta}\frac{-1}{2\pi i}
					\left(
					\frac{G_\varepsilon^+(0)}{x+i\varepsilon}
					-
					\frac{G_\varepsilon^-(0)}{x-i\varepsilon}
					\right)\,dx \\
					&\qquad+
					\lim_{\varepsilon\to+0}\int_{|x| < \delta}\frac{-1}{2\pi i}
					\left(
					\frac{G_\varepsilon^+(x)-G_\varepsilon^+(0)}{x+i\varepsilon}
					-
					\frac{G_\varepsilon^-(x)-G_\varepsilon^-(0)}{x-i\varepsilon}
					\right)\,dx \\
					&\qquad+\lim_{\varepsilon\to+0}\int_{|x| > \delta}\frac{-1}{2\pi i}
					\left(
					\frac{G_\varepsilon^+(x)}{x+i\varepsilon}
					-
					\frac{G_\varepsilon^-(x)}{x-i\varepsilon}
					\right)\,dx 
				\end{align*}
				
				We evaluate each term. For the first term, 
				\begin{align*}
					\MoveEqLeft \lim_{\varepsilon\to+0}\int_{|x|<\delta}\frac{-1}{2\pi i}
					\left(
					\frac{G_\varepsilon^+(0)}{x+i\varepsilon}
					-
					\frac{G_\varepsilon^-(0)}{x-i\varepsilon}
					\right)\,dx = \frac{G_0^+(0)+G_0^-(0)}{2}
				\end{align*}
				
				For the second term, by Lebesgue's dominated convergence theorem,
				\begin{align*}
					\MoveEqLeft\lim_{\varepsilon\to+0}\int_{|x|<\delta}\frac{-1}{2\pi i}
					\left(
					\frac{G_\varepsilon^+(x)-G_\varepsilon^+(0)}{x+i\varepsilon}
					-
					\frac{G_\varepsilon^-(x)-G_\varepsilon^-(0)}{x-i\varepsilon}
					\right)\,dx\\
					& = \int_{|x|<\delta} \lim_{\varepsilon\to+0}\frac{-1}{2\pi i} 
					\left(
					\frac{G_\varepsilon^+(x)-G_\varepsilon^+(0)}{x+i\varepsilon}
					-
					\frac{G_\varepsilon^-(x)-G_\varepsilon^-(0)}{x-i\varepsilon}
					\right)\,dx\\
					& = \int_{|x|<\delta} \frac{-1}{2\pi i} 
					\left(
					\frac{G_0^+(x)-G_0^+(0)}{x}
					-
					\frac{G_0^-(x)-G_0^-(0)}{x}
					\right)\,dx\\
					& = \frac{-1}{2\pi i} \, p.v.\,\int_{|x|<\delta} \left(\frac{G_0^+(x)-G_0^-(x)}{x}	\right) 
					dx.
				\end{align*}
				
				For the third term, again by Lebesgue's dominated convergence theorem,
				\begin{align*}
					\MoveEqLeft\lim_{\varepsilon\to+0}\int_{|x|>\delta}\frac{-1}{2\pi i}
					\left(
					\frac{G_\varepsilon^+(x)}{x+i\varepsilon}
					-
					\frac{G_\varepsilon^-(x)}{x-i\varepsilon}
					\right)\,dx\\
					& = \frac{-1}{2\pi i}\int_{|x|>\delta}
					\frac{G_0^+(x)-G_0^-(x)}{x}
					\,dx
				\end{align*}
				
				By summing them, the claim follows.
			\end{proof}
		
		\vspace{24pt}
		
		We now justify the decomposition \eqref{dis} into the $\delta$-term and the principal value term.
		
		\begin{prp}\label{delpri} Suppose $b>-\frac{1}{2}$, $x(x+t)\neq 0$, and $f\in|y|^{b}\mathcal{S}(\mathbb{R})$,
			\begin{align*}
				\MoveEqLeft e^{tD_{b}}f(x) &\\
				\MoveEqLeft = |x|^{b}\int_{-\infty}^{\infty}\delta\bigl(x+t-y\bigr) \frac{1}{2}\Bigl(\Phi_{b}\bigl(x,y;t+i0\bigr)+\Phi_{b}\bigl(x,y;t-i0\bigr)\Bigr)f(y)\,|y|^{b}dy \\
				\MoveEqLeft \hspace{2pt}+|x|^{b}\int_{-\infty}^{\infty}p.v._{y}\left(\frac{1}{x+t-y}\right)
				\frac{-1}{2\pi i} \Bigl(\Phi_{b}\bigl(x,y;t+i0\bigr)-\Phi_{b}\bigl(x,y;t-i0\bigr)\Bigr)f(y)\,|y|^{b}dy .
			\end{align*}
		\end{prp}
		\begin{proof}
			We apply Lemma~\ref{pvl} after the change of variable $u:=y-(x+t)$.
			Namely, we set
			\[
			G_\varepsilon^\pm(u)
			:=
			|x|^b\,\Phi_b\bigl(x,u+x+t\,;t\pm i\varepsilon\bigr)\,
			f(u+x+t)\,|u+x+t|^b .
			\]
			Then Theorem~\ref{mainA} is rewritten in the form required in Lemma~\ref{pvl}.
			
			It remains to verify the assumptions of Lemma~\ref{pvl}. We recall that
			\[|x|^{b}|y|^{b}\Phi_b(x,y;z) = b\left\{Q_{b-1}\left(\frac{x^2+y^2-z^2}{2|x||y|}\right)-\mathrm{sgn}(xy)Q_{b}\left(\frac{x^2+y^2-z^2}{2|x||y|}\right)\right\}.\]
			We first assume $x-t\neq 0$. The possible singularities of
			$|x|^{b}|y|^{b}\Phi_b(x,y;z)$ come from the Legendre function $Q_\nu$ which occur when
			$\frac{x^2+y^2-z^2}{2xy}=\pm1 \Leftrightarrow y=\pm x\pm t$. As in the proof of Lemma~\ref{del} and by the connection formula in Appendix~\ref{App}, these are at most logarithmic. In particular, at the point $y=x+t$ the leading logarithmic singularity cancels, and hence
			\[
			\frac{G_\varepsilon^\pm(u)-G_\varepsilon^\pm(0)}{u\pm i\varepsilon}
			\]
			is locally dominated by an $L^1$-function near $u=0$.
			
			At the other points $y=\pm x\pm t$ with $y\neq x+t$, the denominator $x+t-y$ does not vanish, so the corresponding singularities are harmless for Lemma~\ref{pvl}; they are still locally $L^1$ because they are at most logarithmic. Finally, the behavior as $y\to\pm\infty$ is controlled by the asymptotics of $\Psi_b$; see Remark~\ref{Psi}. The pointwise convergence assumptions in Lemma~\ref{pvl} are also verified by the above argument. Therefore the hypotheses of Lemma~\ref{pvl} are satisfied in the case $x-t\neq0$.
			
			\vspace{12pt}
			We next assume $x-t=0$. Then the points $y=x-t$ and $y=-x+t$ both collapse to
			$y=0$. Thus, we check this point. The other arguments are the same as in the case $x-t\neq0$. For $z=x\pm i\varepsilon$,
			set
			$W_\varepsilon(y):=\frac{x^2+y^2-z^2}{2xy}
			=
			\frac{y^2+\varepsilon^2\mp 2ix\varepsilon}{2xy}$.
			There exists a sufficiently small $\delta>0$ such that for $|y|<\delta$ and $0<\varepsilon<\delta$, 
			$ |1-W_\varepsilon(y)|\ge \frac12,
			\quad
			|-1-W_\varepsilon(y)|\ge \frac12$.
			In particular, $W_\varepsilon(y)$ stays uniformly away from both $1$ and $-1$. Hence $\Psi_{b}(W_\varepsilon(y))$ is uniformly bounded in this region. This gives the required local $L^{1}$ bound near $y=0$ for the integrand $\frac{1}{x+z-y}|x|^{b}|y|^{b}\Phi_{b}(x,y;z)f(y)= \frac{|x|^{b}|y|^{b}}{x+z-y} (\frac{x^2+y^2-z^2}{2})^{-b}\Psi_{b}(\frac{x^2+y^2-z^2}{2xy})f(y)$. Therefore the hypotheses of Lemma~\ref{pvl} are satisfied also in the case $x-t=0$, and the claim follows.
		\end{proof}
		
		\vspace{12pt}
		
		We set 
		\[K_{b}(x,y;t) := 
		\frac{-1}{2\pi i} |x|^{b}|y|^{b}\lim_{\varepsilon\rightarrow +0}\Bigl(\Phi_{b}\bigl(x,y;t+i\varepsilon\bigr)-\Phi_{b}\bigl(x,y;t-i\varepsilon\bigr)\Bigr),  \]
		
		Since $\Phi_{b}\bigl(x,y; z\bigr)$ is single-valued and holomorphic on $\Bigl\{z\in\mathbb{C}\,\Bigl|\Bigr.\, |z|<\bigl||x|-|y|\bigr| \Bigr\}$ with respect to $z$, \\ 
		$K_{b}(x,y;t)=0$ on $\Bigl\{t\in\mathbb{R}\,\Bigl|\Bigr.\, |t|<\bigl||x|-|y|\bigr| \Bigr\}$.
		
		\vspace{24pt}
		By Propositions~\ref{cos} and \ref{delpri} together with the above argument for support of $K_{b}(x,y;t)$, we obtain Main Theorem~\hyperref[mtB]{B}:
		
		\begin{thm}[Real-Variable Formula]\label{mainB}
			For $b > -\frac{1}{2}$, $x(x+t)\neq 0$ and $f\in |y|^{b}\mathcal{S}(\mathbb{R})$,
			\begin{align*}
				\MoveEqLeft e^{tD_{b}}f(x) &\\
				\MoveEqLeft=\begin{dcases}
					f(x+t)+ \int_{\left||x|-|y|\right|<|t|} p.v._{y}\left(\frac{1}{x+t-y}\right)K_{b}(x,y;t)f(y)\,dy & \text{if}\hspace{14pt} x(x+t)>0,\\
					f(x+t)\cos(b\pi) + \int_{\left||x|-|y|\right|<|t|} p.v._{y}\left(\frac{1}{x+t-y}\right) K_{b}(x,y;t)f(y)\,dy & \text{if}\hspace{14pt} x(x+t) < 0.\\
				\end{dcases}
			\end{align*}
		\end{thm}
		
		\vspace{24pt}
		\subsection{Evaluation of $K_{b}(x,y;t)$}\label{4.6}
		We set 
		
		\[\Theta_{\nu}(u):=  \begin{dcases}
			-\frac{\sin(\nu\pi)}{\pi}Q_{\nu}\left(-u\right) &   \hspace{14pt} \text{if }\hspace{3pt}-\infty < u < -1, \\
			\frac{1}{2}P_{\nu}\left(u\right) &  \hspace{14pt} \text{if }\hspace{3pt}-1 < u < 1, \\
			0  &   \hspace{14pt} \text{if }\hspace{3pt}1 < u <\infty. \\
		\end{dcases}\]
		
		\vspace{12pt}
		We now compute the kernel $K_{b}(x,y;t)$, defined just before Theorem~\ref{mainB}, explicitly, thereby proving Main Theorem~\hyperref[mtC]{C}.
		\begin{thm}[Explicit formula for $K_{b}(x,y;t)$]\label{eva}
			\[ K_{b}(x,y;t) = b\,\mathrm{sgn}(t) \left\{-\Theta_{b-1}\left(\frac{x^{2}+y^{2}-t^{2}}{2|x||y|}  \right)+\mathrm{sgn}(xy)\Theta_{b}\left(\frac{x^{2}+y^{2}-t^{2}}{2|x||y|}  \right) \right\}. \]
			
			That is,
			\begin{align*}
				\MoveEqLeft K_{b}(x,y;t)  =b\,\mathrm{sgn}(t) \\
				\MoveEqLeft\hspace{6pt}\times\begin{dcases}
					0 & \text{if }\hspace{3pt} |t| < ||x|-|y||, \\
					\frac{1}{2}\biggl\{- P_{b-1}\left(\frac{x^{2}+y^{2}-t^{2}}{2|x||y|}  \right) +\mathrm{sgn}(xy)P_{b}\left(\frac{x^{2}+y^{2}-t^{2}}{2|x||y|}  \right) \biggr\} & \text{if }\hspace{3pt} ||x|-|y|| <|t| < |x|+|y|, \\
					-\frac{\sin(b\pi)}{\pi}\left\{Q_{b-1}\left(-\frac{x^{2}+y^{2}-t^{2}}{2|x||y|}  \right)+\mathrm{sgn}(xy)Q_{b}\left(-\frac{x^{2}+y^{2}-t^{2}}{2|x||y|}  \right) \right\} & \text{if }\hspace{3pt} |x|+|y| <|t|.
				\end{dcases} 
			\end{align*}
		\end{thm}
		
		\begin{rmk}
			We set $u=\frac{x^{2}+y^{2}-t^{2}}{2|x||y|}  $. Then, $ u\le  \frac{x^2+y^2}{2|x||y|}$ and
			\begin{align*}
				-\infty < u < -1 &\hspace{14pt}\Leftrightarrow\hspace{14pt} |x|+|y| <|t| <\infty, \\
				-1 < u < 1 &\hspace{14pt}\Leftrightarrow\hspace{14pt} \bigl||x|-|y|\bigr| <|t|<|x|+|y|, \\
				1 < u < \frac{x^2+y^2}{2|x||y|} &\hspace{14pt}\Leftrightarrow\hspace{14pt} 0 <|t|<\bigl||x|-|y|\bigr|.
			\end{align*}
		\end{rmk}
		\vspace{6pt}
		
		\begin{proof}[Proof of Theorem~\ref{eva}]
		By the definition of $\Phi_{b}(x,y;z)$ given at the beginning of Subsection~\ref{4.3},
		\begin{align*}
			\MoveEqLeft K_{b}(x,y;t) = \frac{b}{2\pi i} \Biggl\{\Biggl(Q_{b-1}\left(\frac{x^2+y^2-(t+i0)^2}{2|x||y|}\right)-Q_{b-1}\left(\frac{x^2+y^2-(t-i0)^2}{2|x||y|}\right)\Biggr)\Biggr. \\ 
			&\Biggl.-\mathrm{sgn}(xy)\Biggl(Q_{b}\left(\frac{x^2+y^2-(t+i0)^2}{2|x||y|}\right)-Q_{b}\left(\frac{x^2+y^2-(t-i0)^2}{2|x||y|}\right)\Biggr)\Biggr\}
		\end{align*}
		\vspace{12pt}
		Hence, it suffices to show the following lemma.
		
		\begin{lmm}
			$$
			\frac{-1}{2\pi i} \left(Q_{\nu}\left(u+i0\right)-Q_{\nu}\left(u-i0\right)\right) =\Theta_{\nu}(u) =\begin{dcases}
				-\frac{\sin(\nu\pi)}{\pi}Q_{\nu}\left(-u\right) &   \hspace{14pt} \text{if }\hspace{3pt}-\infty < u < -1, \\
				\frac{1}{2}P_{\nu}\left(u\right) &  \hspace{14pt} \text{if }\hspace{3pt}-1 < u < 1, \\
				0  &   \hspace{14pt}\text{if }\hspace{3pt} 1 < u <\infty. \\
			\end{dcases}$$
		\end{lmm}
		
		\vspace{6pt}
		\begin{proof} 
				\begin{enumerate}
					\item (When $1<u<\infty$) \\
					Since $Q_{\nu}(w)$ is holomorphic on $\mathbb{C}\setminus(-\infty,1]$,
					the boundary values from above and below coincide. This proves the claim.
					\item (When $-\infty<u<-1$)\\ By the identity
					\begin{align*}
						Q_{\nu}(-w)=-e^{-\nu\pi i}Q_{\nu}(w) \hspace{24pt}\text{if }\hspace{8pt}\mathrm{Im}(w)\,<\,0 \,\\
						Q_{\nu}(-w)=-e^{\nu\pi i}Q_{\nu}(w) \hspace{24pt}\text{if } \hspace{8pt}\mathrm{Im}(w)\,>\,0 \,
					\end{align*}
					the claim follows. The above identity follows from the definition of $Q_{\nu}(w)$; see also \cite[8.736, items 5 and 6]{MR3307944}, for a reference.
					\item (When $-1<u<1$) \\
					Recalling that
					
					\[Q_{\nu}(w)=\frac{1}{2}\frac{\Gamma\left(\frac{\nu+1}{2}\right)\Gamma\left(\frac{\nu+2}{2}\right)}{\Gamma(\nu+3/2)}w^{-(\nu+1)} {}_{2}F_{1}\left(\begin{matrix}
						\frac{\nu+1}{2}, \frac{\nu+2}{2} \\ \nu+\frac{3}{2} \\
					\end{matrix};\frac{1}{w^2}\right), \]
					
					we use the following connection formula for the Gauss hypergeometric function in the case $\gamma=\alpha+\beta$; a proof is given in Appendix~\ref{App}:
					\begin{align*}
						\MoveEqLeft {}_{2}F_{1}\left(\begin{matrix}
							\alpha, \beta \\ \gamma \\
						\end{matrix}; w\right) = \frac{\Gamma(\gamma)}{\Gamma(\alpha)\Gamma(\beta)}\left\{-\log\left(1-w\right) {}_{2}F_{1}\left(\begin{matrix}
							\alpha, \beta \\ 1 \\
						\end{matrix};1-w\right) \right.&\\
						&\left.- \sum_{m=0}^{\infty} \frac{(\alpha)_{m}(\beta)_{m}}{m!} \bigl(\psi(\alpha+m)+\psi(\beta+m)-2\psi(1+m) \bigr)\frac{(1-w)^{m}}{m!}\right\}.
					\end{align*}
					Here $\psi(w):=\frac{\Gamma'(w)}{\Gamma(w)}$ is the digamma function.
				
					\vspace{6pt}
					Then, 
					\[\lim_{u\to 1} \frac{Q_{\nu}(u+i0)-Q_{\nu}(u-i0)}{2} = -\frac{\pi i}{2}\]
					
					Since singular points of the Legendre differential equation $(1-x^{2})y''-2x y'+\nu(\nu+1)y=0$ are regular and $P_{\nu}(1)=1$, 
					
					$$ \frac{Q_{\nu}(u+ i0)- Q_{\nu}(u- i0)}{2} =- \frac{\pi i}{2} P_{\nu}(u)\hspace{36pt}(-1 < u < 1 ).$$
					
					This proves the claim.
					
					(See also \cite[8.705 and 8.732, item 5]{MR3307944}, for a reference.)
					
				\end{enumerate}
				
			\end{proof}
		\end{proof}
		
		\section{Appendix}\label{App}
		\subsection{A Connection Formula for ${}_2F_{1}(\alpha,\beta;\gamma;z)$ in the Case $\gamma=\alpha+\beta$}

		\begin{lmm}[Connection Formula for ${}_2F_{1}(\alpha,\beta;\gamma;z)$ when $\gamma-\alpha-\beta=0$]
			When $\gamma=\alpha+\beta$, 
			\begin{align*}
				\MoveEqLeft {}_{2}F_{1}\left(\begin{matrix}
					\alpha, \beta \\ \gamma \\
				\end{matrix};w\right) = \frac{\Gamma(\gamma)}{\Gamma(\alpha)\Gamma(\beta)}\left\{-\log\left(1-w\right) {}_{2}F_{1}\left(\begin{matrix}
					\alpha, \beta \\ 1 \\
				\end{matrix};1-w\right) \right.&\\
				&\left.- \sum_{m=0}^{\infty} \frac{(\alpha)_{m}(\beta)_{m}}{m!} \bigl(\psi(\alpha+m)+\psi(\beta+m)-2\psi(1+m) \bigr)\frac{(1-w)^{m}}{m!}\right\}.
			\end{align*}
			Here $\psi(w):=\frac{\Gamma'(w)}{\Gamma(w)}$ is the digamma function.
		\end{lmm}
	
		\begin{proof}
			We set 
			\[{}_{2}\widetilde{F}_{1}\left(\begin{matrix}
				\alpha, \beta \\ \gamma \\
			\end{matrix};w\right):=\frac{1}{\Gamma(\gamma)} {}_{2}F_{1}\left(\begin{matrix}
			\alpha, \beta \\ \gamma \\
			\end{matrix};w\right).\]
			
			Then, the following connection formula holds:
			\[ {}_{2}\widetilde{F}_{1}\left(\begin{matrix}
				\alpha, \beta \\ \gamma \\
			\end{matrix};w\right) =\frac{\pi}{\sin(\pi\delta)}\left(\frac{{}_{2}\widetilde{F}_{1}\left(\begin{matrix}
				\alpha, \beta \\ -\delta+1 \\
			\end{matrix};1-w\right)}{\Gamma(\delta+\alpha)\Gamma(\delta+\beta)}-(1-w)^{\delta} \frac{{}_{2}\widetilde{F}_{1}\left(\begin{matrix}
			\delta+\alpha, \delta+\beta \\ \delta+1 \\
		\end{matrix};1-w\right)}{\Gamma(\alpha)\Gamma(\beta)}\right),\]
	
		where $\delta := \gamma-\alpha-\beta$. (See \cite{NIST:DLMF}, 15.8.4, for a reference.) The above formula holds for $0<w<1$, and extends by analytic continuation. Taking the limit as $\delta\to0$, we obtain the claim.
		\end{proof}
		
		\vspace{12pt}
			\section*{Acknowledgements}
			The author would like to express his gratitude to his supervisor, Professor Toshiyuki Kobayashi, for his continuous support and encouragement. 
			This research was supported partially by JSPS KAKENHI Grant Number JP24KJ0937 and Forefront Physics and Mathematics Program to Drive Transformation (FoPM), a World-leading Innovative Graduate Study (WINGS) Program, The University of Tokyo.

	\bibliography{paper2}	

\begin{thebibliography}{BK{\O}12}

\bibitem[BK{\O}12]{MR2956043}
Salem Ben~Sa{\"i}d, Toshiyuki Kobayashi, and Bent {\O}rsted.
\newblock Laguerre semigroup and {D}unkl operators.
\newblock {\em Compos. Math.}, 148(4):1265--1336, 2012.

\bibitem[{\relax DLMF}]{NIST:DLMF}
{\it NIST Digital Library of Mathematical Functions}.
\newblock \url{https://dlmf.nist.gov/}, Release 1.2.5 of 2025-12-15.
\newblock F.~W.~J. Olver, A.~B. {Olde Daalhuis}, D.~W. Lozier, B.~I. Schneider,
  R.~F. Boisvert, C.~W. Clark, B.~R. Miller, B.~V. Saunders, H.~S. Cohl, and
  M.~A. McClain, eds.

\bibitem[Dun89]{MR951883}
Charles~F. Dunkl.
\newblock Differential-difference operators associated to reflection groups.
\newblock {\em Trans. Amer. Math. Soc.}, 311(1):167--183, 1989.

\bibitem[Dun92]{MR1199124}
Charles~F. Dunkl.
\newblock Hankel transforms associated to finite reflection groups.
\newblock In {\em Hypergeometric functions on domains of positivity, {J}ack
  polynomials, and applications ({T}ampa, {FL}, 1991)}, volume 138 of {\em
  Contemp. Math.}, pages 123--138. Amer. Math. Soc., Providence, RI, 1992.

\bibitem[GR15]{MR3307944}
I.~S. Gradshteyn and I.~M. Ryzhik.
\newblock {\em Table of integrals, series, and products}.
\newblock Elsevier/Academic Press, Amsterdam, eighth edition, 2015.
\newblock Translated from the Russian, Translation edited and with a preface by
  Daniel Zwillinger and Victor Moll.

\bibitem[KM07]{MR2401813}
Toshiyuki Kobayashi and Gen Mano.
\newblock The inversion formula and holomorphic extension of the minimal
  representation of the conformal group.
\newblock In {\em Harmonic analysis, group representations, automorphic forms
  and invariant theory}, volume~12 of {\em Lect. Notes Ser. Inst. Math. Sci.
  Natl. Univ. Singap.}, pages 151--208. World Sci. Publ., Hackensack, NJ, 2007.

\bibitem[Wat44]{MR10746}
G.~N. Watson.
\newblock {\em A {T}reatise on the {T}heory of {B}essel {F}unctions}.
\newblock Cambridge University Press, Cambridge; The Macmillan Company, New
  York, 1944.

\end{thebibliography}
	\bibliographystyle{alpha}

\end{document}